\documentclass[10pt]{amsart}
\theoremstyle{plain}
\usepackage{amsmath, amsfonts, amsthm, amssymb,amscd}
\usepackage{dbnsymb1}
\usepackage{graphics}
\usepackage{color}
\usepackage{epsfig}
\usepackage[all]{xy}
\usepackage{hyperref}
\title{Mutation invariance of Khovanov homology over $\Ftwo$}
\author{Stephan M. Wehrli}
\address{Institut de Math\'ematiques de Jussieu;
Universit\'e Paris 7 (Denis Diderot); 175 rue du Chevaleret,
bureau 7B3; 75013 Paris; France}
\email{wehrli@math.jussieu.fr}

\theoremstyle{plain}
\newtheorem{theorem}{Theorem}[section]
\newtheorem{lemma}[theorem]{Lemma}
\newtheorem{proposition}[theorem]{Proposition}
\newtheorem{corollary}[theorem]{Corollary}

\theoremstyle{definition}
\newtheorem{definition}[theorem]{Definition}
\newtheorem{remark}[theorem]{Remark}

\newcommand{\R}{\ensuremath{\mathbb{R}}}
\newcommand{\Z}{\ensuremath{\mathbb{Z}}}
\newcommand{\C}{\ensuremath{\mathbb{C}}}
\newcommand{\dd}{\partial_{\bullet}}
\newcommand{\Ftwo}{\mathbb{F}_2}
\newcommand{\Mat}{\operatorname{Mat}}
\newcommand{\Kom}{\operatorname{Kom}}
\newcommand{\Cobdl}[1]{\mathcal{C}ob(#1)_{\bullet/\ell}}

\begin{document}
\bibliographystyle{halpha}

\begin{abstract}
We prove that Khovanov homology and Lee homology
with coefficients in $\Ftwo=\Z/2\Z$ are
invariant under component-preserving link mutations.
\end{abstract}
\maketitle


\section{Introduction}
Khovanov homology is a refinement of the Jones polynomial
\cite{jones-1985} which was discovered by Mikhail Khovanov
\cite{MR1740682} in the year 1999, and which was subsequently
generalized through the work Eun Soo Lee \cite{lee-2002}
and Dror Bar-Natan \cite{barnatan-2005-9}.
In 2003, the author \cite{wehrli-2003} discovered
a series of examples of mutant links \cite{MR0258014}
with different (integer coefficient) Khovanov homology.
Despite this discovery, the question whether there are
mutant knots with different Khovanov homology remained open.
In this paper, we partially answer this question. We prove:

\begin{theorem}\label{thm:main}
The graded homotopy type of $Kh(L)$ is invariant
under component-preserving link mutation.
\end{theorem}

In this theorem, $Kh(L)$ stands for a variant
of Bar-Natan's formal Khovanov bracket \cite{barnatan-2005-9},
which generalizes
both Khovanov homology with $\Ftwo$ coefficients
and Lee homology \cite{lee-2002} with $\Ftwo$ coefficients.
To prove the theorem, we will employ
an argument that was outlined by Bar-Natan in 2005
\cite{barnatan-2005}. While Bar-Natan's
argument had some gaps, the author realized that these
gaps can be filled if one works over $\Ftwo$ coefficients.
In 2007, the author presented a complete proof of
Theorem~\ref{thm:main} at
the `Knots in Washington XXIV' conference in Washington D.C.,
and at the `Link homology and Categorification' conference
in Kyoto \cite{wehrli-2007}.

More recently, Jonathan Bloom \cite{bloom-2009-3} discovered
an alternative and completely independent proof of mutation invariance.
Bloom's proof has the advantage that it works not only over
$\Ftwo$ coefficients, but rather extends to
a proof showing that the odd version
of the integer coefficient Khovanov homology
(defined as in \cite{ozsvath-2007})
is invariant under arbitrary
link mutations. On the other hand, the proof given in this paper
has the advantage that it also implies that \emph{Lee homology}
with $\Ftwo$ coefficients is invariant under component-preserving
link mutation.

The paper is organized as follows.
In Section~\ref{s:mutation},
we show that every \emph{component-preserving link mutation}
can be realized by a finite sequence of \emph{crossed $z$-mutations}
and \emph{isotopies}. In Section~\ref{s:bracket}, we
introduce the variant of the formal Khovanov bracket
that we will use throughout this paper. This variant
takes values in a category whose morphisms are
formal $\Ftwo$-linear combinations of properly embedded
$2$-cobordisms, decorated by finitely many distinct dots,
and considered up to some relations.
In Section~\ref{s:tools}, we
discuss algebraic operations for
manipulating the dots that appear in a decorated cobordism,
and
in Section~\ref{s:proof}, we use these operations
to prove that our variant of the formal Khovanov bracket
is invariant under crossed $z$-mutation.


\section{Conway mutation}\label{s:mutation}
Let $U\subset\R^2$ be the closure of a domain in $\R^2$,
and let $P\subset\partial U$ a finite subset of $\partial U$.
A \emph{tangle} above $(U,P)$ is
a properly embedded compact
$1$-manifold $\mathcal{T}\subset U\times\R$
with $\partial\mathcal{T}=P\times\{0\}$.
To  represent a tangle above $(U,P)$,
we use a plane diagram $T\subset U$ with $\partial T=P$.
In the case where $T$ is a plane diagram of
a tangle $\mathcal{T}$ above the unit disk
$U=\mathcal{D}:=\{z\in\C=\R^2\colon |z|\leq 1\}$,
and $P$ is the set
$P:=\{a,b,c,d\}\subset\partial\mathcal{D}$ where
$a,b,c,d$ are the points
$\exp(i\pi n/4)\in\C=\R^2$ for $n=1,3,5,7$ (in this order),
then we denote by $R_x(T),R_y(T),
R_z(T)$ the plane diagrams of the tangles
obtained by rotating $\mathcal{T}\subset\mathcal{D}\times\R\subset\R^3$
by $180^{\circ}$ around the $x$-, $y$- and $z$-axis respectively.

\begin{figure}[!h]
$$
\begin{array}{c}
   \includegraphics[height=3cm]{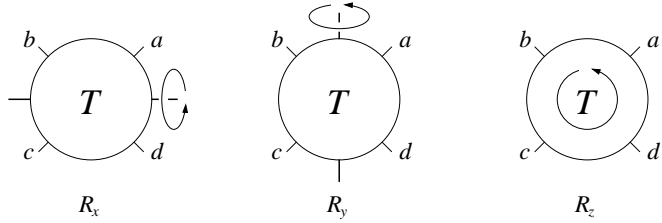}
\end{array}
$$
\caption{Rotations $R_x$, $R_y$, $R_z$.}
\end{figure}

Let $\mathcal{D}^c:=\R^2\setminus\operatorname{Int}(\mathcal{D})$
and $P:=\{a,b,c,d\}$.
If $\mathcal{T}$ is a tangle over $(\mathcal{D},P)$,
and $\mathcal{T}'$
is a tangle over $(\mathcal{D}^c,P)$,
then the union $\mathcal{T}\cup\mathcal{T}'$
is a link $\mathcal{L}=\mathcal{T}\cup\mathcal{T}'\subset \R^2\times\R=\R^3$.

\begin{definition}\label{def:mutation}
Two links $\mathcal{L}$ and $\mathcal{L}'$
are called \emph{elementary Conway mutants of each other}
\cite{MR0258014}
if there is a rotation $R\in\{R_x,R_y,R_z\}$
and two tangle diagrams $T\subset\mathcal{D}$
and $T'\subset\mathcal{D}^c$ with
$\partial T=\partial T'=P$ and
such that $T\cup T'$ is a diagram for $\mathcal{L}$
and $R(T)\cup T'$ is a diagram for $\mathcal{L}'$.
Depending on whether $R=R_x$, $R_y$ or $R_z$,
we say that the diagrams $T\cup T'$ and
$R(T)\cup T'$ are related by \emph{$x$-, $y$-} or
\emph{$z$-mutation}.
\end{definition}

\begin{remark}
If $\mathcal{L}$ and $\mathcal{L}'$
are oriented, then we require that
$T\cup T'$ is a diagram for $\mathcal{L}$,
and $R(T)\cup T'$ or $R(-T)\cup T'$ (whichever of
the two is oriented consistently)
is a diagram for $\mathcal{L}'$.
\end{remark}

\begin{definition}\label{def:crossed}
We say that $T\cup T'$
and $R(T)\cup T'$ are related by a
\emph{crossed mutation} if the tangle
corresponding to $T'\subset\mathcal{D}^c$
has \emph{crossed connectivity}, i.e. if one of
its arcs has endpoints at
$\{a\}\times\{0\}$ and $\{c\}\times\{0\}$,
and the other arc has endpoints at
$\{b\}\times\{0\}$ and $\{d\}\times\{0\}$.
\end{definition}

\begin{definition}
We say that $\mathcal{L}=\mathcal{T}\cup\mathcal{T}'$
and $\mathcal{L}'=R(\mathcal{T})\cup\mathcal{T}'$
are related by a \emph{component-preserving mutation}
if the union $R(\alpha)\cup\alpha'$ is a connected
component of $\mathcal{L}'$ if and only if the union
$\alpha\cup\alpha'$ is a connected component of $\mathcal{L}$,
for any two arc components $\alpha\subset\mathcal{T}$ and
$\alpha'\subset\mathcal{T}'$.
\end{definition}

\begin{figure}
$$
\begin{array}{c}
   \includegraphics[height=1.9cm]{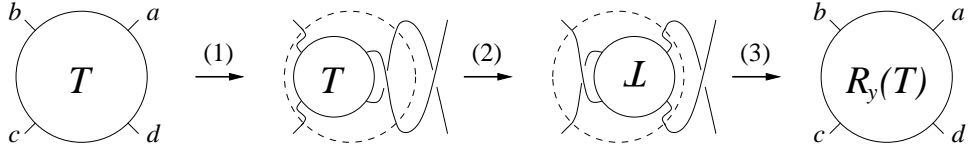}
\end{array}
$$
\caption{\label{fig:yviaz}
Decomposing a $y$-mutation into three steps:
(1) a Reidemeister move of type II;
(2) a $z$-mutation along the dashed circle;
(3) an isotopy in $\R^3$ that rotates $\mathcal{T}$ around
the $x$-axis
and thus untwists the crossings on either side of $\mathcal{T}$.
}
\end{figure}

The following lemma allows us to reduce Theorem~\ref{thm:main}
to Proposition~\ref{prop:main} below.

\begin{lemma}\label{lem:mutation}
Let $\mathcal{L}$ and $\mathcal{L}'$ be two
links that are related by
component-preserving mutation,
and let $D$ be a planar diagram of $\mathcal{L}$
and $D'$ a planar diagram of $\mathcal{L}'$.
Then $D$ can be transformed into $D'$
by a sequence of Reidemeister moves
and crossed $z$-mutations.
\end{lemma}

\begin{proof}
It is easy to see that the three different types
of mutation ($x$-, $y$- and $z$-mutation)
are topologically equivalent. Indeed, Figure~\ref{fig:yviaz}
shows how a $y$-mutation can be obtained by
performing a Reidemeister move of type II, followed
by a $z$-mutation, followed by an
isotopy in $\R^3$, and analogously, an $x$-mutation
can be reduced to a $z$-mutation.
Thus, we can assume without
loss of generality that $D$ and $D'$ are
related by a $z$-mutation,
i.e. $D=T\cup T'$ and $D'=R_z(T)\cup T'$ for suitable
tangle diagrams $T\subset\mathcal{D}$ and
$T'\subset\mathcal{D}^c$. If $T'$ has crossed
connectivity, then there is nothing to prove,
and if $T$ has crossed connectivity,
then we can interchange the roles
of $T$ and $T'$ by applying a planar isotopy
which moves $T'$ into $\mathcal{D}$ and $T$ out of
$\mathcal{D}$. Thus, we only need to care about the
case where neither $T$ nor $T'$ has crossed connectivity.
In this case, either $T$ or $T'$  must have horizontal connectivity
(i.e., represent a tangle that contains an arc with
endpoints at $\{a\}\times\{0\}$ and $\{b\}\times\{0\}$),
for otherwise the mutation would not be component-preserving.
After interchanging the roles of $T$ and $T'$ if necessary,
we can assume that $T'$ has horizontal connectivity.
But then the $z$-mutation in Step~(2) of Figure
is a crossed $z$-mutation, and hence Figure
shows that $D=T\cup T'$ can be transformed into
$R_y(T)\cup T'$ by Reidemeister moves
and a crossed $z$-mutation. A similar argument shows
$R_y(T)\cup T'$ can be transformed into $R_y(T)\cup R_x(T')$
by Reidemeister moves and a crossed $z$-mutation,
and since $R_z=R_x\circ R_y$, the latter diagram is isotopic
to $R_x\left(R_y(T)\cup R_x(T')\right)=R_z(T)\cup T'=D'$,
whence the proof is complete.
\end{proof}

The following proposition is the main result of this paper.
Its proof will be given in Section~\ref{s:proof}.

\begin{proposition}\label{prop:main}
If two link diagrams are related by a crossed $z$-mutation,
then their formal Khovanov brackets are isomorphic.
\end{proposition}


\section{Bar-Natan's formal Khovanov bracket}
\label{s:bracket}
In this section, we briefly review the definition of Bar-Natan's
formal Khovanov bracket. For more details, we refer the
reader to \cite{barnatan-2005-9}.

\subsection{Chain complexes and chain maps in pre-additive categories}
Let $\mathcal{C}$ be a pre-additive category.
To $\mathcal{C}$, one can associate an additive
category $\Mat(\mathcal{C})$, called
the \emph{matrix extension} or \emph{additive closure}
of $\mathcal{C}$ and defined as follows.
An object of $\Mat(\mathcal{C})$ is a finite
tuple $(O_1,\ldots,O_m)$ of objects $O_i\in\mathcal{C}$
(where $m$ can be any non-negative integer).
A morphism
$F\colon (O_1,\ldots,O_n)\rightarrow (O'_1,\ldots,O'_m)$
is a matrix $F=(F_{ij})$ of morphisms
$F_{ij}\in\operatorname{Hom}_{\mathcal{C}}(O_j,O'_i)$.
The composition of two morphisms $F=(F_{ik})$ and $G=(G_{kl})$
is modelled on ordinary
matrix multiplication: $(F\circ G)_{ij}:=\sum_kF_{ik}\circ G_{kj}$.
Direct sums are defined by concatenation:
$(O_1,\ldots,O_n)\oplus(O'_1,\ldots,O'_m):=(O_1,\ldots,O_n,O'_1,\ldots,O'_m)$.
By identifying an object $O\in\mathcal{C}$ with the $1$-tuple
$(O)\in\Mat(\mathcal{C})$,
one can embed $\mathcal{C}$ into $\Mat(\mathcal{C})$ as
a full subcategory. In particular, one can write every
object $(O_1,\ldots,O_m)\in\Mat(\mathcal{C})$ as
a direct sum $(O_1,\ldots,O_m)=\bigoplus_{i=1}^mO_i$.

\begin{definition}
A \emph{bounded chain complex} in $\mathcal{C}$
is a pair $C=(C^*,d^*)$, where $C^*=\{C^i\}_{i\in\Z}$
is a sequence
of objects $C^i\in\Mat(\mathcal{C})$,
such that $C^i=0$ for $|i|\gg 0$,
and $d^*=\{d^i\}_{i\in\Z}$
is sequence of morphisms $d^i \colon C^i\rightarrow C^{i+1}$
such that $d^{i+1}\circ d^i=0$ for all $i\in\Z$.
\end{definition}

\begin{definition}
A \emph{chain map} $F\colon(C_1^*,d_1^*)\rightarrow(C_2^*,d_2*)$
is a sequence of morphisms $F^i\colon C_1^i\rightarrow C_2^i$
such that $F^{i+1}\circ d_1^i=d_2^i\circ F^i$ for all $i\in\Z$.
\end{definition}

We denote by $\Kom(\mathcal{C})$ the category whose objects
are bounded chain complexes in $\mathcal{C}$ and
whose morphisms are chain maps.

\begin{remark}
If $\mathsf{F}\colon\mathcal{C}_1\rightarrow\mathcal{C}_2$
is an additive functor
between two pre-additive categories
$\mathcal{C}_1$ and $\mathcal{C}_2$, then $\mathsf{F}$
can be extended to an
additive functor $\mathsf{F}\colon\Mat(\mathcal{C}_1)\rightarrow\Mat(\mathcal{C}_2))$
by setting
$\mathsf{F}((O_1,\ldots,O_m)):=(\mathsf{F}(O_1),\ldots,\mathsf{F}(O_m))$
and $\mathsf{F}(F):=(\mathsf{F}(F_{ij}))$
for every object $(O_1,\ldots,O_m)\in\Mat(\mathcal{C}_1)$ and
every morphism $F=(F_{ij})$. Similarly, $\mathsf{F}$ can
be extended
to an additive functor
$\mathsf{F}\colon\Kom(\mathcal{C}_1)\rightarrow\Kom(\mathcal{C}_2)$ by setting
$\mathsf{F}((C^*,d^*))^i:=(\mathsf{F}(C^i),\mathsf{F}(d^i))$
and $\mathsf{F}(F^*)^i:=\mathsf{F}(F^i)$. In this paper,
we make no distinction between the notation for
the functor $\mathsf{F}\colon\mathcal{C}_1\rightarrow\mathcal{C}_2$
itself, and the notation for the extensions of $\mathsf{F}$.
\end{remark}

\subsection{Decorated cobordisms}\label{subs:decorated}
In the following, $U$ is the closure of
a domain in $\R^2$, and $P$ is a finite subset of $\partial U$.

Let $O_1,O_2\subset U$ be two properly embedded unoriented
compact $1$-submanifolds in $U$ with
$\partial O_1=\partial O_2=P$.
A \emph{cobordims} between $O_1$ and $O_2$
is a compact properly embedded unoriented surface
$S\subset U\times [0,1]$
whose bottom boundary is $O_1$ and whose top boundary
is $O_2$, and whose intersection with
$(\partial U)\times [0,1]$ consists of
the vertical segments $P\times [0,1]$.
A \emph{decorated cobordism} is a cobordism
decorated by
finitely many (possibly zero) distinct points
or \emph{dots}, which lie in the interior of $S$.
Let $DC(O_1,O_2)_{\bullet}$ be the
set of isotopy classes of
decorated cobordisms between $O_1$ and $O_2$.
Moreover, let $DC(O_1,O_2)_{\bullet/\ell}$ be
the quotient of the $\Ftwo$-vector space spanned
the elements of $DC(O_1,O_2)_{\bullet}$ modulo the
following \emph{local relations}, called respectively
the \emph{sphere relation}, the \emph{dot relation} and the
\emph{neck-cutting relation:}
\begin{figure}[!h]
\begin{equation*}
\begin{array}{l}
  \text{(S)}\quad
  \begin{array}{c}
    \includegraphics[height=1cm]{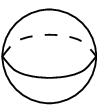}
  \end{array}\hspace{-2mm}\,\sqcup S\,=\,0,
  \qquad\qquad\;
  \text{(D)}\quad
  \begin{array}{c}
    \includegraphics[height=1cm]{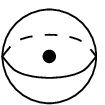}
  \end{array}\hspace{-2mm}\,\sqcup S\,=\,S,
\\
  \text{(N)}\quad
  \begin{array}{c}\includegraphics[height=10mm]{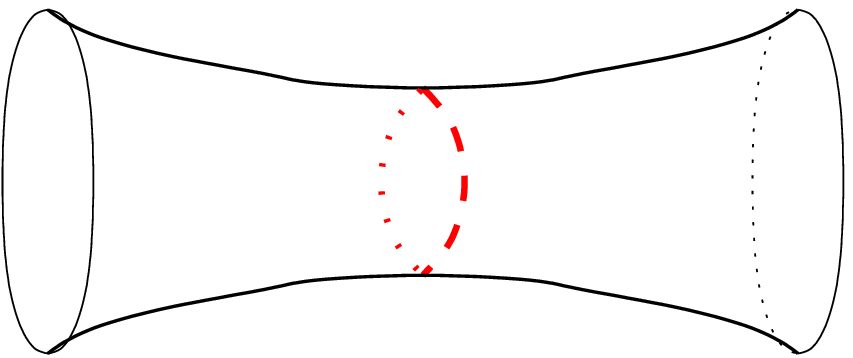}\end{array}
\,=\,\begin{array}{c}
\includegraphics[height=10mm]{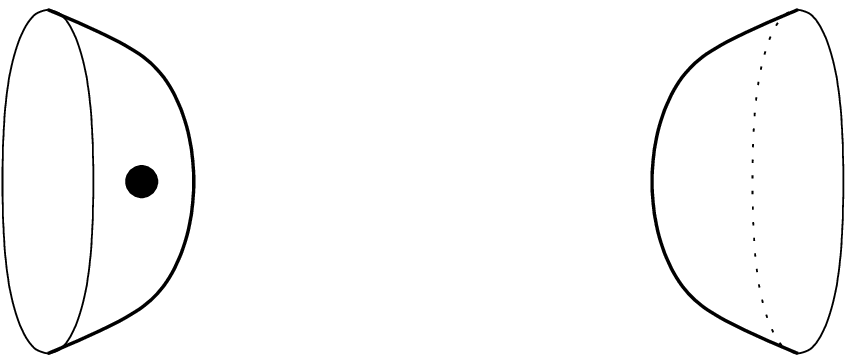}\end{array}\,
+\,\begin{array}{c}
\includegraphics[height=10mm]{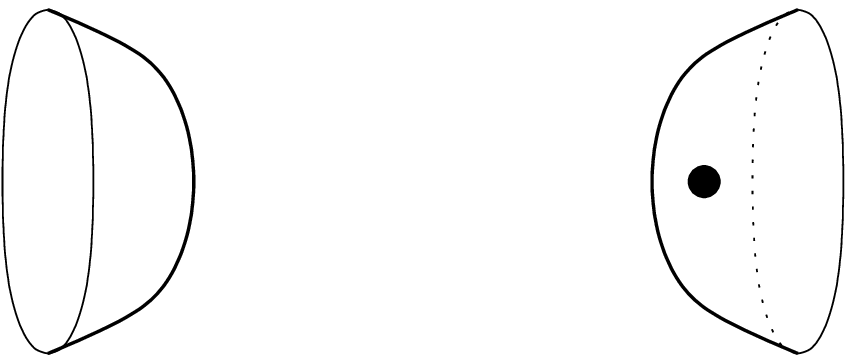}\end{array}.
\end{array}
\end{equation*}
\caption{\label{fig:relations} Local relations in $DC(O_1,O_2)_{\bullet/\ell}$.}
\end{figure}

In the first two relations, $S$ stands for an arbitrary decorated cobordism,
and in the third relation, the three pictures stand for three decorated
cobordisms, which are identical everywhere except in a small
ball $B^3\subset U\times [0,1]$ where the differ as shown.
Using the above relations, one can deduce the important
\emph{double dot relation:}
\begin{figure}[!h]
$$
  \text{(DD)}\quad
  \begin{array}{c}
    \includegraphics[height=1cm]{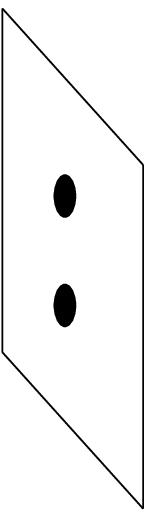}
  \end{array}\hspace{-2mm}\,=\,
  \begin{array}{c}
    \includegraphics[height=1cm]{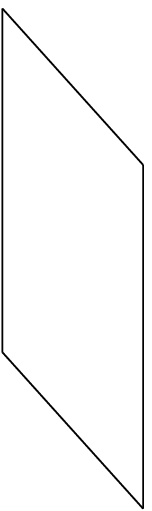}
  \end{array}\sqcup\,t
$$
\caption{\label{fig:ddrelation} The double dot relation.}
\end{figure}

In the (DD) relation, $t$ stands for a $2$-sphere decorated by exactly three
dots. Thus, the (DD) relation says that we can remove any pair
of dots lying on the same
component of a decorated cobordism, at
the expense of adding a $2$-sphere decorated by
exactly three dots. We can endow
$DS(O_1,O_2)_{\bullet/\ell}$ with the structure an $\Ftwo[t]$-module 
by defining $t^nS$ to be the disjoint union of $S$ with $n$
disjoint copies of $t$.

\begin{definition} Let $\Cobdl{U,P}$ be the pre-additve category
whose objects are unoriented
properly embedded compact $1$-manifolds
$O\subset U$ with $\partial O=P$, and whose morphism sets are the
$\Ftwo$-vector spaces
$DC(O_1,O_2)_{\bullet/\ell}$.
Composition of morphisms $S_1\colon O_1\rightarrow O_2$
and $S_2\colon O_2\rightarrow O_3$ is given by stacking
$S_2$ on top of $S_1$.
\end{definition}

Let $\Mat(U,P):=\Mat(\Cobdl{U,P})$ and
$\Kom(U,P):=\Kom(\Cobdl{U,P})$.

\subsection{Quantum grading}\label{subs:grading}
To incorporate the \emph{quantum grading} (or
\emph{$j$-grading}) of Khovanov homology,
one has to redefine the objects of $\Cobdl{U,P}$
as being pairs $(O,n)$ where $O\subset U$
is a properly embedded compact $1$-manifold with $\partial O=P$
as before, and $n$ is an integer. A
morphism $S\colon(O_1,n_1)\rightarrow(O_2,n_2)$ is
given by a morphism $S\colon O_1\rightarrow O_2$, i.e.
by an element $S\in DC(O_1,O_2)_{\bullet/\ell}$.
The \emph{quantum degree} of a morphism
is defined by:
$$\deg(S):=e(S)-2d(S)+n_2-n_1\,,$$
where $e(S):=\chi(S)-|P|/2$ is the Euler measure of $S$,
and $d(S)$ is the number of dots on $S$.
Let $\Cobdl{U,P}^0$ denote the category which
has the same objects as $\Cobdl{U,P}$, but whose morphisms
$S\colon (O_1,n_1)\rightarrow(O_2,n_2)$
are required to satisfy $\deg(S)=0$.
Let $\Mat(U,P)^0:=\Mat(\Cobdl{U,P}^0)$
and $\Kom(U,P)^0:=\Kom(\Cobdl{U,P}^0)$.
For each integer $m$, let
$\{m\}$ denote the degree shift functor given
by $(O,n)\{m\}:=(O,m+n)$. Identifying
$(O,0)$ with $O$, we will henceforth write $O\{n\}$ instead of $(O,n)$.

\subsection{Formal Khovanov bracket}\label{subs:bracket}
Now let $T\subset U$ be a tangle diagram with $\partial T=P$.
Let $\chi$ be the set of crossings of $T$
and $\{0,1\}^{\chi}$ the set
of all maps $\epsilon\colon\chi\rightarrow\{0,1\}$.
A crossing $c\in\chi$ (looking like: $\slashoverback$)
can be resolved in two possible
ways, $\smoothing$ and $\hsmoothing$,
called its \emph{$0$-resolution}
and its \emph{$1$-resolution}, respectively.
Given $\epsilon\in\{0,1\}^{\chi}$, denote by $T_{\epsilon}$
the crossingless tangle diagram obtained from
$T$ by replacing every $c\in\epsilon^{-1}(0)$
by its $0$-resolution, and every $c\in\epsilon^{-1}(1)$
by its $1$-resolution.
For $\epsilon,\epsilon'\in\{0,1\}^{\chi}$
and $c\in\chi$, we will write $\epsilon<_c\epsilon'$ iff
$\epsilon$ and $\epsilon'$ satisfy
$\epsilon(c)=0$ and $\epsilon'(c)=1$,
and $\epsilon(c')=\epsilon'(c')$ for all $c'\in\chi$ with $c'\neq c$.
For such $\epsilon,\epsilon'$,
there is a preferred cobordism
$S_{\epsilon'\epsilon}\colon T_{\epsilon}\rightarrow T_{\epsilon'}$
containing no dots, such that
$S_{\epsilon',\epsilon}\cap (\operatorname{Nbd}(c)\times [0,1])$
is a saddle cobordism between
$\smoothing$ and $\hsmoothing$, and
$S_{\epsilon',\epsilon}\setminus(\operatorname{Nbd}(c)\times [0,1])$
is the identity cobordism.
For $\epsilon,\epsilon'\in\{0,1\}^{\chi}$
and $c\in\chi$, let
$(d_c)_{\epsilon'\epsilon}\colon T_{\epsilon}\rightarrow T_{\epsilon'}$
be the morphism defined by
$(d_c)_{\epsilon'\epsilon}:=S_{\epsilon',\epsilon}$
if $\epsilon<_c\epsilon'$, and $(d_c)_{\epsilon'\epsilon}:=0$
otherwise. Let
$d_{\epsilon'\epsilon}:=\sum_{c\in\chi}(d_c)_{\epsilon'\epsilon}$
and $|\epsilon|:=|\epsilon^{-1}(1)|=\sum_{c\in\chi}\epsilon(c)$.
Suppose $T$ is oriented and let $n_+$ ($n_-$) be
the number of positive (negative) crossings in $T$.
If $\epsilon$ and $\epsilon'$ satisfy
$|\epsilon|=i+n_-$ and $|\epsilon'|=i+1+n_-$
for an $i\in\Z$,
then we set $d^i_{\epsilon'\epsilon}:=d_{\epsilon'\epsilon}$.

\begin{definition}\label{def:kh}
The \emph{formal Khovanov bracket} of $T$ is
the chain complex $Kh(T):=(Kh(T)^*,\\ d^*)\in
\Kom(U,P)^0$ defined by
$Kh(T)^i:=\bigoplus_{|\epsilon|=i+n_-} T_{\epsilon}\{i+n_+-2n_-\}$ and
$d^i:=(d^i_{\epsilon'\epsilon})$.
\end{definition}

Definition~\ref{def:kh} is justified by the following lemma:

\begin{lemma}\label{lem:dsquare}
$d^{i+1}\circ d^i=0$ for all $i\in\Z$.
\end{lemma}
\begin{proof}
Ignoring differentials and gradings for a moment,
we can identify $Kh(T)$ with the object
$Kh(T)=\bigoplus_{\epsilon\in\{0,1\}^{\chi}}T_{\epsilon}
\in\Mat(U,P)$. We can then identify the
differential in $Kh(T)$
with the endomorphism $d:=(d_{\epsilon'\epsilon})$
of $Kh(T)\in\Mat(U,P)$ (with
$d_{\epsilon'\epsilon}$ defined as above).
For $c\in\chi$, let $d_c$ be the endomorphism 
of $Kh(T)\in\Mat(U,P)$
defined by $d_c:=((d_c)_{\epsilon'\epsilon})$.
We have $d_c\circ d_c=0$
because for any three elements
$\epsilon,\epsilon',\epsilon''\in\{0,1\}^{\chi}$,
at least one of the two matrix entries
$(d_c)_{\epsilon''\epsilon'}$ and
$(d_c)_{\epsilon'\epsilon}$ is equal to zero.
We also have $d_c\circ d_{c'}=d_{c'}\circ d_c$
for all $c,c'\in\chi$ because distant saddles can
be time-reordered by isotopy. Since $d=\sum_{c\in\chi}d_c$,
this implies $d\circ d=0$, and thus the lemma follows.
\end{proof}

The following theorem was proved by Bar-Natan
\cite{barnatan-2005-9}.

\begin{theorem}\label{thm:tangleinvariant}
The graded homotopy type of $Kh(T)$ is a tangle invariant.
\end{theorem}

\subsection{Relation with Khovanov homology and Lee homology}
If $T$ is a link diagram (i.e. $\partial T=\emptyset$),
then the formal Khovanov bracket of $T$
refines both  the $\Ftwo$-coefficient
\emph{Khovanov homology}~\cite{MR1740682} 
and the $\Ftwo$-coefficient
\emph{Lee homology}~\cite{lee-2002} of $T$.
Indeed, let $\operatorname{Hom}(\emptyset,-)$
be the functor which maps an object $O\in\Cobdl{U,\emptyset}$
to the graded morphism set
$\operatorname{Hom}(\emptyset,O)$,
regarded as a graded $\Ftwo[t]$-module via the (DD) relation.
Then the $\Ftwo$-coefficient
Khovanov homology of $T$ is the homology of the
chain complex $\mathsf{F}_{Kh}(Kh(T))$, where
$\mathsf{F}_{Kh}(-):=\operatorname{Hom}(\emptyset,-)\otimes_{t=0}\Ftwo$,
and the $\Ftwo$-coefficient
Lee homology of $T$ is the homology of the chain complex
of $\mathsf{F}_{Lee}(Kh(T))$, where
$\mathsf{F}_{Lee}(-):=\operatorname{Hom}(\emptyset,-)
\otimes_{t=1}\Ftwo$.

\subsection{Tensor products}
In this subsection, we describe a special case of
the `categorified planar algebra' structure of $Kh(T)$
that was introduced in \cite[Section~5]{barnatan-2005-9}.
Assume that we have the following situation:
\begin{itemize}
\item $U'$ and $U''$ are the closures of two disjoint domains in $\R^2$
and $U:=U'\cup U''$
\item $P_1$ and $P_2$ are finite subsets of
$(\partial U')\setminus U''$ and $(\partial U'')\setminus U'$, respectively.
\item $P_0$ is a finite subset of $U'\cap U''$.
\item $P':=P_0\cup P_1$ and $P'':=P_0\cup P_2$ and $P:=P_1\cup P_2$.
\end{itemize}
In this situation, there is a natural functor
$$
\Cobdl{U',P'}\times\Cobdl{U'',P''}\longrightarrow\Cobdl{U,P}
$$
which takes a pair of objects $(O',O'')$ (or morphisms $(S',S'')$)
to the union $O'\cup O''$ (or $S'\cup S''$). We write this
functor as a tensor product, and we extend it to a functor
$\Mat(U',P')\times\Mat(U'',P'')\rightarrow\Mat(U,P)$ by declaring
that the tensor product distributes over direct sums,
i.e.
$(O'_1\oplus O'_2)\otimes(O''_1\oplus O''_2):=
(O'_1\otimes O''_2)\oplus (O'_1\otimes O''_2)\oplus
(O'_2\otimes O''_2)\oplus (O'_2\otimes O''_2)$
and $(F'\otimes F'')_{i\otimes k,j\otimes l}=F'_{ij}\otimes F''_{kl}$.
Given two chain complexes $C'\in\Kom(U',P')$
and $C''\in\Kom(U'',P'')$, we define
$C'\otimes C''\in\Kom(U,P)$ to be the
chain complex whose underlying object is the tensor product
$C'\otimes C''\in\Mat(U,P)$,
and whose differential
is the endomorphism (in $\Mat(U,P)$) given by
$$
d_{C'\otimes C''}:=d_{C'}\otimes 1_{C'} + 1_{C''}\otimes d_{C''}
$$ 
where $d_{C'}$, $d_{C''}$, $1_{C'}$, $1_{C''}$ are
the differentials and the identity morphisms of $C'$ and $C''$,
respectively. As for the gradings, it is understood that both
the homological grading and the quantum grading are
additive under tensor products. The following theorem
was shown (in greater generality) in
\cite[Section~5]{barnatan-2005-9}.

\begin{theorem}\label{thm:tensor}
Let $T'\subset U'$ and $T''\subset U''$ be tangle diagrams
with $\partial T'=P'$ and $\partial T''=P''$.
Then $Kh(T'\cup T'')$ is canonically isomorphic
to $Kh(T')\otimes Kh(T'')$.
\end{theorem}

\subsection{Delooping}\label{subs:delooping}
Let `$\bigcirc$' denote
the connected $1$-manifold consisting
of a single circle. More generally, let `$\bigcirc^n $'
denote the $1$-manifold consisting of $n$ disjoint circles,
and let
$\emptyset\{1\}$ and $\emptyset\{-1\}$ denote
degree-shifted copies of the empty $1$-manifold.
The following lemma is well-known (see e.g.
\cite[Lemma~4.1]{barnatan-2006}).

\begin{lemma}\label{lem:delooping}
The objects $\bigcirc$ and $\emptyset\{1\}\oplus\emptyset\{-1\}$
are isomorphic in $\Mat(U,\emptyset)^0$.
\end{lemma}
\begin{proof}
Let $V:=\emptyset\{1\}\oplus\emptyset\{-1\}$, and let
$G\colon\bigcirc\rightarrow V$ and
$H\colon V\rightarrow\bigcirc$ be the morphisms given by
the matrices
$(G_{11},G_{21})^t$ and $(H_{11},H_{12})$,
where $G_{11},G_{21},H_{11},H_{12}$
are cobordisms homeomorphic
to disks, with $G_{21}$ and $H_{11}$ containing no dots,
and $G_{11}$ and $H_{12}$ containing a single
dot each.
Using the local relations shown in Figure~\ref{fig:relations},
one can easily check that $G\circ H$ and $H\circ G$ are
the identity morphism of $V$ and $\bigcirc$, respectively.
\end{proof}

Let $\mathcal{C}\subset \Cobdl{U,P}$ be
the full subcategory
containing of all objects of the form $O\{n\}$,
where $O$ is a $1$-manifold without closed
components, and $n\in\Z$ is an arbitrary integer
(in fact, we will henceforth
drop the $\{n\}$ from the notation).
Note that
every object $O\in\Cobdl{U,P}$ can be written in the form
$O=O'\otimes\bigcirc^n$, where $O'\in\mathcal{C}$
and $n\geq 0$, and the tensor product `$\otimes$' denotes
a disjoint union. (This notation is consistent with the
one used in the previous subsection for $P_0=\emptyset$).
By applying the isomorphism
$G\colon\bigcirc\rightarrow V$ defined in the proof of
Lemma~\ref{lem:delooping} repeatedly
to each circle in $O=O'\otimes\bigcirc^n$,
we can define a functor which sends the object
$O\in\Mat(U,P)$ to an isomorphic object in $\Mat(\mathcal{C})$.
Formally, this functor is defined as follows.

\begin{definition}\label{def:delooping}
The \emph{delooping functor}
$\mathsf{D}\colon\Mat(U,P)\rightarrow\Mat(\mathcal{C})$
sends an object $O=O'\otimes\bigcirc^n$ (with
$O'\in\mathcal{C}$) to the object $\mathsf{D}(O):=O'\otimes V^{\otimes n}$,
and a morphism $S\colon O'_1\otimes\bigcirc^{n_1}
\rightarrow O'_2\otimes\bigcirc^{n_2}$ to the
morphism $\mathsf{D}(S):=(1\otimes G^{\otimes n_2})\circ S\circ (1\otimes H^{\otimes n_1})$
where $V$ and $G,H$ are as in the proof of Lemma~\ref{lem:delooping},
and $1$ stands for the identity morphism of either $O'_1$ or $O'_2$.
\end{definition}


\section{Operations involving dots}\label{s:tools}
In this section, we define algebraic
operations for manipulating
the dots that decorate a decorated cobordism.

\subsection{Dot multiplication}\label{subs:dmultiplication}
Let $U$ be the closure of a domain in $\R^2$ and
$P$ be a finite subset of $\partial U$.
Let $O\subset U$ be an object of the pre-additive category
$\Cobdl{U,P}$ defined in Subsection~\ref{subs:decorated},
and let $p\in O$ be an arbitrary point on $O$.

\begin{definition}\label{def:dmultiplication}
The \emph{dot multiplication map} is the endomorphism
$X_p\colon O\rightarrow O$
given by the cobordism $O\times [0,1]$, decorated
by a single dot lying in the interior of the
segment $\{p\}\times [0,1]\subset O\times[0,1]$.
If $p$ is a point of $\partial O=P$, then we move the
dot slightly into the interior of $O\times [0,1]$,
so that the result is a decorated cobordism in
the sense of Subsection~\ref{subs:decorated}.
\end{definition}

If $O_1,O_2\subset U$ are two objects of $\Cobdl{U,P}$
containing a point $p\in O_1\cap O_2$, and $S\colon O_1\rightarrow O_2$
is a decorated cobordism commuting with $X_p$,
then we define
$$x_pS:=X_p\circ S=S\circ X_p\,.$$
The above definitions extend to
$\Mat(U,P)$ as follows. Let $O=(O_1,\ldots,O_m)$
be an object in $\Mat(U,P)$ and $p\in\bigcap O_i$.
Then the dot multiplication map $X_p\colon O\rightarrow O$
is the endomorphism
whose off-diagonal entries are zero and whose diagonal
entry $(X_p)_{ii}$ is the decorated cobordism $x_p(O_i\times [0,1])$.
Similarly, if $F\colon O\rightarrow O'$ is a morphism commuting
with $X_p\colon O\rightarrow O'$ for a point
$p\in\bigcap O_i\cap\bigcap O'_j$, then we define
$x_pF:=X_p\circ F=F\circ X_p$.

\begin{definition} The \emph{endpoint ring} $\Ftwo[P]$ is the
commutative
polynomial ring with coefficients in $\Ftwo$ in formal variables
$x_p$, one for each $p\in P$.
\end{definition}

Since every morphism in $\Cobdl{U,P}$
contains the segment $\{p\}\times[0,1]$ and hence commutes with
$X_p$ for all $p\in P$, the endpoint ring
$\Ftwo[P]$ acts on morphism sets
of $\Cobdl{U,P}$ (or $\Mat{U,P}$)
by $x_p\cdot S:=x_pS=X_p\circ S=S\circ X_p$.

\subsection{Dot derivation}
Let $O_1,O_2\subset U$ be two compact embedded
$1$-manifolds with $\partial O_1=\partial O_2=P$,
and let $S\in DC(O_1,O_2)_{\bullet}$ be a
decorated cobordism containing $m\geq 0$ dots.

\begin{definition}\label{def:dderivative}
The \emph{derivative of $S$ with respect to the dot} is the sum
$$
\dd S := S_1 +\ldots + S_m\,\in\, DC(O_1,O_2)_{\bullet/\ell}\,,
$$
where $S_i$ is the decorated cobordism obtained from $S$ by
removing the $i$th dot.
\end{definition}

\begin{lemma}\label{lem:dd}
The map
$\dd\colon S\mapsto \dd S$ descends to a linear
endomorphism of $DC(O_1,O_2)_{\bullet/\ell}$.
\end{lemma}

\begin{proof}
We have to check that $\dd$ is compatible with the local
relation shown in Figure~\ref{fig:relations}. Applying
$\dd$ to the two sides of the (S) relation
yields zero on both sides, and so there is nothing
to prove in this case. Applying $\dd$ to the (D) relation
yields zero
on the right-hand side and an undecorated sphere on the left-hand
side. But an undecorated sphere is equivalent
to zero by the (S) relation, whence $\dd$ is also compatible with
the (D) relation. Compatibility with the (N) relation follows because $\dd$
applied to the left-hand side of (N) gives zero, and $\dd$ applied
to the right-hand side of (N) yields a sum of two identical term,
which is zero because we are working with $\Ftwo$ coefficients.
\end{proof}

The above lemma implies that $\dd$ acts on the morphism
sets of $\Cobdl{U,P}$, and the following lemma says that
$\dd$ satisfies Leibniz' rule with respect to composition
of morphisms.

\begin{lemma}\label{lem:ddproperties}
We have $\dd (S\circ S') = (\dd S)\circ S'+S\circ\dd S'$.
\end{lemma}
\begin{proof} Obvious from the definition of $\dd$.
\end{proof}

\begin{corollary}\label{cor:ddcommute}
If $S$ satisfies $S\circ S=0$,
then $S$ commutes with $\dd S$.
\end{corollary}
\begin{proof}
Since coefficients are in $\Ftwo$ and since
$\dd$ satisfies Leibniz' rule by Lemma~\ref{lem:ddproperties},
we can write the commutator of $S$ with $\dd S$ as
$
[S,\dd S] = S\circ\dd S + (\dd S)\circ S
= \dd (S\circ S)
$, and
thus the corollary follows.
\end{proof}

We extend $\dd$ to morphisms of $\Mat(U,P)$ (or $\Kom(U,P)$)
by setting $\dd(F_{ij}):=(\dd F_{ij})$. It is easy to see that
Lemma~\ref{lem:ddproperties} and Corollary~\ref{cor:ddcommute}
remain true for this extended version of $\dd$.

\begin{remark}
Note that $\dd$ raises the quantum degree by $2$
and satisfies $\dd\circ\dd=0$ (again we are using
that coefficients are in $\Ftwo$). Therefore,
the subcategory $\Cobdl{U,P}^{ev}\subset\Cobdl{U,P}$
which has the same objects as $\Cobdl{U,P}$ but
whose morphisms are required to have even quantum
degree (i.e. $\operatorname{deg}(S)\in 2\Z$)
becomes a differential graded category when
equipped with the derivation $\dd$.
\end{remark}

\subsection{Dot rotation}\label{subs:drotation}
In this subsection, we assume that $U=\mathcal{D}$
is the closed unit disk in $\R^2$ and $P\subset\partial U$
is the set $P=\{a,b,c,d\}$ defined in Section~\ref{s:mutation}.
As in Section~\ref{s:mutation}, we denote by $R_z$ the
self map of $\mathcal{D}\times[0,1]\subset\R^3$
given by $180^{\circ}$ rotation around the $z$-axis.
Since $R_z(P)=P$, the rotation $R_z$ acts on
objects and morphisms of $\Cobdl{\mathcal{D},P}$
by sending an object $O\subset\mathcal{D}$
to the rotated object $R_z(O)$, and a morphism
$S\subset\mathcal{D}\times [0,1]$ to
the rotated morphism $R_z(S)$. Since this action is compatible
with the composition of morphisms, it defines a functor
$$
R_z\colon\Cobdl{\mathcal{D},P}\longrightarrow\Cobdl{\mathcal{D},P}
$$
The goal of this subsection is to re-express
this functor in terms of the algebraic operations introduced
in the previous two subsections. To do this, we
first define
$$
r_z\colon\Ftwo[P]\longrightarrow\Ftwo[P]
$$
to be the ring automorphism
induced by mapping $x_p\in \Ftwo[P]:=\Ftwo[x_a,x_b,x_c,x_d]$
to  $r_z(x_p):=x_{R_z(p)}\in\Ftwo[P]$
for all $p\in P$. Explicitly, $r_z$ exchanges $x_a$ with $x_c$
and $x_b$ with $x_d$.
The following lemma is obvious.

\begin{lemma}\label{lem:drotation}
$R_z(f S)=r_z(f) R_z(S)$ for
every morphism $S$ in $\mathcal{C}$ and every $f\in\Ftwo[P]$.
\end{lemma}

Now let $\mathcal{C}$
be the full subcategory of $\Cobdl{\mathcal{D},P}$
containing all objects without closed components,
and let $\mathsf{D}\colon\Mat(\mathcal{D},P)\rightarrow\Mat(\mathcal{C})$
be the delooping functor defined as in Subsection~\ref{subs:delooping}.
The subcategory $\mathcal{C}$ contains two preferred objects:
$O_0:=[a,d]\cup [b,c]$ and $O_1:=[a,b]\cup [c,d]$, where
$[p,q]\subset\mathcal{D}$ denotes the straight line segment
connecting the points $p,q\in P$. Let $\mathcal{C}'$
be the full subcategory of $\mathcal{C}$ over the objects
$O_0$ and $O_1$. (More precisely, $\mathcal{C}'$ contains
all objects that are of the form
$O\{n\}$ where $O\in\{O_0,O_1\}$ and $\{n\}$ is
a grading shift by an arbitrary $n\in\Z$).
Since every object in $\mathcal{C}$ is isotopic rel. boundary
(and hence
isomorphic in $\mathcal{C}$) to exactly one of the
two objects $O_0$ and $O_1$, we can define a natural
functor $\mathsf{S}\colon\mathcal{C}\rightarrow\mathcal{C}'$
by sending $O\in\mathcal{C}$ to
$O_0$ or $O_1$, whichever of the two is isomorphic to $O$.
Of course, this functor extends to $\Mat(\mathcal{C})$ (or
$\Kom(\mathcal{C})$), and we will also write $\mathsf{S}$
for this extended functor.

\begin{definition}\label{def:enhanced}
The \emph{enhanced delooping functor}
is the composition $\mathsf{D}':=\mathsf{S}\circ\mathsf{D}$.
\end{definition}

\begin{lemma}\label{lem:drotationclear}
$\mathsf{D}'(O)$ is isomorphic to $O$
for every $O\in\Mat(\mathcal{D},P)$ (or $\Kom(\mathcal{D},P$)).
\end{lemma}
\begin{proof}
Clear from the definitions of $\mathsf{D}$ and $\mathsf{S}$.
\end{proof}

Since $O_0$ and $O_1$ are invariant under
rotation by $180^{\circ}$, the functor
$R_z$ acts as the identity on the set
$\operatorname{Ob}(\mathcal{C}')=\{O_0,O_1\}$.

\begin{definition}\label{def:drotation}
The \emph{dot rotation functor} is the endofunctor
$R_{\bullet}\colon\mathcal{C}'\rightarrow\mathcal{C}'$
which acts as the identity on the set
$\operatorname{Ob}(\mathcal{C}')=\{O_0,O_1\}$ and
which takes a morphism $S$ to the morphism
$$
R_{\bullet}(S):=S+(x_a+x_c)\dd S\,.
$$
\end{definition}

\begin{lemma}\label{lem:drotation1}
$R_z(S)=R_{\bullet}(S)$ for every morphism
$S$ in $\mathcal{C}'$.
\end{lemma}

\begin{proof}
Let $S\subset\mathcal{D}\times [0,1]$ be
a decorated cobordism representing a  morphism in $\mathcal{C}'$.
Using the local relations shown in Figures~\ref{fig:relations}
and \ref{fig:ddrelation}, we can write as $S=S'\sqcup t^n=:t^nS'$,
where $t^n$ is a disjoint union of $n\geq 0$ two-spheres,
each or them decorated by exactly three dots, and
$S'$ is a decorated cobordism whose
every component is homeomorphic to a disk and
decorated by at most one dot.
Let $S''$ be the undecorated cobordism underlying $S'$.
Then $S''$ has to be either a saddle cobordism
or one of the two identity cobordisms $O_0\times [0,1]$
or $O_1\times [0,1]$ (as these are the only undecorated
cobordisms in $\mathcal{C}'$ that have the property that
all of their connected compoents are homeomorphic to disks).
In particular, $S''$ is invariant under $R_z$
and has at most two connected components.
Moreover, every connected component of $S''$ contains
at least one of the two segments
$\{a\}\times [0,1]$ or $\{c\}\times [0,1]$,
and this means that we can write $S'$ as
$S'=x_a^{n_a}x_c^{n_c}S''$ for appropriate $n_a,n_c\in\{0,1\}$
(where e.g. $x_ax_cS''$ denotes the decorated cobordism
$X_a\circ X_c\circ S''$
as in Subsection~\ref{subs:dmultiplication}).
Writing $f$ for the monomial
$t^nx_a^{n_a}x_c^{n_c}\in\Ftwo[t,x_a,x_c]$
and using Lemma~\ref{lem:drotation}, we obtain:
$$
R_z(S)=r_z(f)R_z(S'')=r_z(f) S''
$$
One can easily check that $\dd t=0$, and
since $S''$ contains no dots, we also have $\dd S''=0$.
Using Lemma~\ref{lem:ddproperties}
we therefore obtain $\dd S=\dd(f S'')=(\partial f)S''$, where
$\partial\colon\Ftwo[t,x_a,x_c]\rightarrow\Ftwo[t,x_a,x_c]$ is
the $\Ftwo[t]$-linear map defined by
$\partial:=\partial/\partial x_a +\partial/\partial x_c$.
Thus:
$$
R_{\bullet}(S)=\left[f + (x_a + x_c)(\partial f)\right]S''
$$
Comparing the above expressions for $R_z(S)$ and $R_{\bullet}(S)$,
we see that it suffices to prove the equivalence
$r_z(f)\equiv f + (x_a + x_c)(\partial f)$
modulo local relations. We do
this by case by case analysis: if $f=t^n$,
then $r_z(f)=f$ and $\partial f=0$, so the result follows.
If $f=t^nx_a$, then $r_z(f)=t^nx_c$ and $\partial f=t^n$, so
$r_z(f)=t^nx_c=2t^nx_a + t^nx_c=f+ (x_a +x_c)(\partial f)$;
the case $f=t^nx_c$ is analogous.
Finally, if $f=t^nx_ax_c$, then $r_z(f)=f$ and
$$(x_a + x_c)(\partial f)S''=t^n(x_a +x_c)^2S''=t^n(x_a^2 + x_c^2)S''=2t^{n+1}S''=0\,,$$
where we have used the (DD) relation and the fact that
coefficients are in $\Ftwo$.
\end{proof}

\begin{corollary}\label{cor:drotation}
$R_{\bullet}(\mathsf{D}'(O))$ is isomorphic to $R_z(O)$
for all $O\in\Mat(\mathcal{D},P)$ (or $\Kom(\mathcal{D},P)$).
\end{corollary}

\begin{proof}
The functors $\mathsf{D}$ and $\mathsf{S}$
are clearly equivariant under the rotation $R_z$,
and hence $\mathsf{D}'=\mathsf{S}\circ\mathsf{D}$
commutes with $R_z$. Using Lemmas~\ref{lem:drotationclear}
and \ref{lem:drotation1},
we thus obtain $R_z(O)\cong\mathsf{D}'(R_z(O))=R_z(\mathsf{D}'(O))
=R_{\bullet}(\mathsf{D}'(O))$.
\end{proof}

\subsection{Dot migration}\label{subs:dotmigration}
Let $T'$ be a tangle diagram in
$\mathcal{D}^c:=\{z\in\C=\R^2\colon|z|\geq 1\}$
with $\partial T'=P=\{a,b,c,d\}$. Assume that
$T'$ has crossed connectivity as in Proposition~\ref{prop:main},
i.e. that it represents a
a tangle
$\mathcal{T}'\subset\mathcal{D}^c\times\R$ which contains
an arc connecting the endpoints $\{a\}\times\{0\}$
and $\{c\}\times\{0\}$.
Let $\alpha\subset T'$ be the projection of this arc,
and let $c_1,\ldots,c_m\subset\alpha$
be the crossings of $T'$ along $\alpha$, enumerated
in the order shown in Figure~\ref{fig:dotmigration}.
\begin{figure}[!h]
$$
\begin{array}{c}
   \includegraphics[height=2.8cm]{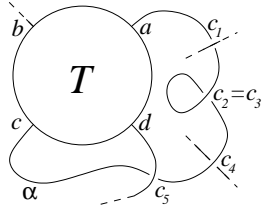}
\end{array}
$$
\caption{\label{fig:dotmigration} Crossings $c_1,\ldots,c_m$ along the arc $\alpha\subset T'$.}
\end{figure}

For $k=2,\ldots,m$, let $e_k\subset\alpha$ be the connected
component of $\alpha\setminus\bigcup_kc_k$ which lies
between $c_{k-1}$ and $c_k$, and let $p_k\in e_k$ denote
the midpoint of $e_k$. Put $p_1:=a$ and $p_{m+1}:=c$.

\begin{definition}\label{def:dmultiplication1}
Let $X_1,\ldots,X_{m+1}$ be the endomorphisms of
$\bigoplus_{i\in\Z}Kh(T')^i\in\Mat(\mathcal{D}^c,P)$
defined by $X_k:=X_{p_k}$, where $X_{p_k}$
is the dot multiplication map
defined in Subsection~\ref{subs:dmultiplication}.
\end{definition}

As explained in the proof of Lemma~\ref{lem:dsquare},
the differential in $Kh(T')$ can be regarded
as an endomorphism $d$ of the object $\bigoplus_{i\in\Z}Kh(T')^i
\in\Mat(\mathcal{D}^c,P)$, and this endomorphism can
be written as a sum $d=\sum_{c\in\chi}d_c$.
Recall that the matrix entries $d_{\epsilon'\epsilon}$
and $(d_c)_{\epsilon'\epsilon}$
are either zero or given by a saddle cobordism
$S_{\epsilon'\epsilon}\subset\mathcal{D}^c\times [0,1]$.
Let $r\colon\mathcal{D}^c\times[0,1]\rightarrow\mathcal{D}^c\times[0,1]$
be the reflection along $\mathcal{D}^c\times\{1/2\}$, and let $d_k:=d_{c_k}$.

\begin{definition}\label{def:dmigration}
The \emph{dot migration homotopies} $h_1,\ldots,h_m$
are the endomorphisms of
$\bigoplus_{i\in\Z}Kh(T')^i
=\bigoplus_{\epsilon\in\{0,1\}^{\chi}} T'_{\epsilon}
\in\Mat(\mathcal{D}^c,P)$
defined by $h_k:=d_k^{\dagger}$
where $(d_k^{\dagger})_{\epsilon'\epsilon}:=r((d_k)_{\epsilon\epsilon'})$.
\end{definition}

Arguing as in the proof of Lemma~\ref{lem:dsquare}, one can easily show:

\begin{lemma}\label{lem:homotopyproperties0}
We have
\begin{enumerate}
\item $h_k\circ h_k=0$,
\item $h_k\circ h_l=h_l\circ h_k$,
\item $h_k\circ d_c=d_c\circ h_k$,
\end{enumerate}
for all $k,l=1,\ldots,m$ and all crossings $c\neq c_k$.
\end{lemma}

The next lemma says that $h_k$ is a homotopy between $X_k$ and
$X_{k+1}$.

\begin{lemma}\label{lem:homotopy}
$d\circ h_k + h_k\circ d = X_k + X_{k+1}$.
\end{lemma}

\begin{proof}
Since $d=\sum_{c\in\chi}d_c$ and since
$h_k$ commutes with $d_c$ for all $c\in\chi$ with $c\neq c_k$,
we have $d\circ h_k + h_k\circ d=d_k\circ h_k +h_k\circ d_k$,
and so it is enough to prove $d_k\circ h_k +h_k\circ d_k=X_k+X_{k+1}$.
Since this is a purely local equation, we can restrict ourselves
to the case where $k=1$ and $\chi=\{c_1\}$, i.e. where $T'$
has only one crossing. Then $Kh(T')=T'_0\oplus T'_1$ (here
we ignore the homological grading and the quantum grading),
where $T'_0$ and $T'_1$ are the crossingless diagrams obtained
by replacing the crossing $c_1$ ($= \slashoverback$) by
its $0$-resolution ($\smoothing$)
and its $1$-resolution ($\hsmoothing$),
respectively. We can regard the differential $d=d_1$ in $Kh(T')$
as an endomorphism of the object $T'_0\oplus T'_1\in\Mat(\mathcal{D}^c,P)$.
As such, it is given by a $2\times 2$ matrix, whose
only non-zero entry is $d_{10}=S_{10}$, where $S_{10}$
is a saddle cobordism (as in Subsection~\ref{subs:bracket}).
Similarly, the homotopy $h:=h_1$ is given by a $2\times 2$-matrix
whose only non-zero entry is the saddle cobordism
$h_{01}=r(S_{10})$. Thus, $(h\circ d)_{00}=r(S_{10})\circ S_{10}$
and $(d\circ h)_{11}= S_{10}\circ r(S_{10})$, and
all other matrix entries in $h\circ d$ and $d\circ h$ are zero.
The cobordism $r(S_{10})\circ S_{10}$ is a composition
of two `opposite' saddle cobordisms, and it is easy to
see that such a composition results in a cobordism looking
like the identity cobordism $T'_0\times [0,1]$, except that the
two components of $\smoothing\times [0,1]$ are connected by a tube.
Applying the (N) relation to this tube, we obtain
$$
(h\circ d)_{00}=r(S_{10})\circ S_{10}= (x_1 + x_2)(T'_0\times [0,1])
= (x_1+x_2)1_{00}
$$
where $1_{00}$ is the identity morphism of $T'_0$. Similarly,
we obtain $(d\circ h)_{11} =(x_1 + x_2)1_{11}$ where $1_{11}$
is the identity morphism of $T'_1$.
Thus, $d\circ h + h\circ d=(x_1 + x_2)1=X_1 +X_2$ as desired.
\end{proof}

\begin{lemma}\label{lem:homotopyproperties}
$h_k\circ d\circ h_k =0$.
\end{lemma}
\begin{proof}
By the previous lemma, we have
$d\circ h_k=h_k\circ d + X_k + X_{k+1}$, and inserting this
into $h_k\circ d\circ h_k$, we obtain
$h_k\circ d\circ h_k= h_k\circ h_k \circ d + h_k\circ (X_k + X_{k+1})$.
The first term on the right-hand side vanishes because $h_k\circ h_k=0$,
and to see that the second term vanishes, we can assume
that $T'$ consists of a single crossing, i.e. $k=1$
and $\chi=\{c_1\}$ as in the proof of the previous lemma.
Then $(h_1)_{01}=r(S_{10})$ 
(as in the proof of the previous lemma),
and since the cobordism $r(S_{10})$ has only one connected component, we
have $x_1 r(S_{10})=x_2 r(S_{10})$, whence $h_1\circ X_1=h_1\circ X_2$.
Using that coefficients are in $\Ftwo$, we get
$h_1\circ (X_1 + X_2)= 2h_1\circ X_1=0$.
\end{proof}


\section{Proof of Proposition~\ref{prop:main}}
\label{s:proof}
In this section, we use the notations of Section~\ref{s:mutation}.
assume that the hypotheses of
Proposition~\ref{prop:main} are satisfied. Thus,

In particular, $T$ denotes a tangle
diagram in the unit disk $\mathcal{D}\subset\R^2$,
and $T'$ a tangle diagram
in $\mathcal{D}^c:=\R^2\setminus\operatorname{Int}(\mathcal{D})$.
The endpoints of $T$ and $T'$ lie in the set
$\partial T=\partial T'=P=\{a,b,c,d\}\subset\partial\mathcal{D}$.
As in Proposition~\ref{prop:main}, we assume that
$T'$ represents a tangle $\mathcal{T}'\subset\mathcal{D}^c\times\R\subset\R^3$
which has crossed connectivity, i.e. contains an arc connecting
the endpoints $\{a\}\times\{0\}$ and $\{c\}\times\{0\}$.
We also assume that the mutation is a $z$-mutation, i.e.
that it consists in replacing $T$ by $R_z(T)$.
Let $L:=T\cup T'$ and $L':=R_z(T)\cup T'$ denote the link
diagrams before and after mutation.
Using the tensor product theorem (Theorem~\ref{thm:tensor}),
we can write the formal Khovanov brackets of $L$ and $L'$
as:
$$
Kh(L)=Kh(T)\otimes Kh(T')\qquad\text{and}\qquad
Kh(L')=Kh(R_z(T))\otimes Kh(T')
$$
Let $\mathcal{C}'\subset\Cobdl{\mathcal{D},P}$ be the
full subcategory generated by the two objects
$O_0:=[a,d]\cup[b,c]$ and $O_1:=[a,b]\cup[c,d]$ where
$[p,q]\subset\mathcal{D}$ denotes the straight line
segment connecting the points $p,q\in P$ as in
Subsection~\ref{subs:drotation}.
Let $\mathsf{D}'\colon\Mat(\mathcal{D},P)\rightarrow\Mat(\mathcal{C}')$
denote the enhanced delooping functor (Definition~\ref{def:enhanced})
and $R_{\bullet}\colon\Mat(\mathcal{C}')\rightarrow\Mat(\mathcal{C}')$
the dot rotation functor (Definition~\ref{def:drotation}).
By Lemma~\ref{lem:drotationclear}, $Kh(T)$
is isomorphic to $\mathsf{D}'(Kh(T))$,
and hence $Kh(L)$ is isomorphic to the complex
$$
A:=\mathsf{D}'(Kh(T))\otimes Kh(T')
$$
Since the construction of $Kh(T)$ is equivariant with respect to
the rotation $R_z$, we have $Kh(R_z(T))=R_z(Kh(T))$.
Moreover, Corollary~\ref{cor:drotation} implies that
$R_z(Kh(T))$ is isomorphic to $R_{\bullet}(\mathsf{D}'(Kh(T)))$,
and hence $Kh(L')$ is isomorphic to the complex
$$
B:=R_{\bullet}(\mathsf{D}'(Kh(T)))\otimes Kh(T')
$$

To prove Proposition~\ref{prop:main}, it is now enough to show $A$ is 
isomorphic to $B$. By definition, $R_{\bullet}$ acts
as the identity on the
set $\operatorname{Ob}(\mathcal{C}')=\{O_0,O_1\}$, and so
we have $A=B$ if we ignore the differentials in $A$
and $B$ (i.e. if we just consider the
objects
$\bigoplus_{i\in\Z} A^i$ and
$\bigoplus_{i\in\Z} B^i$ of $\Mat(\R^2,\emptyset)$
instead of the actual complexes $A=(A^*,d^*_A)$
and $B=(B^*,d^*_B)$).
The differentials in $A$ and $B$ are given by
$$
d_A=\delta\otimes 1 + 1\otimes d\qquad\text{and}\qquad
d_B=R_{\bullet}(\delta)\otimes 1 + 1\otimes d\,,
$$
where $\delta$ is the differential
in $\mathsf{D}'(Kh(T))$ and $d$ is the differential in
$Kh(T')$,
and $1$ stands for an identity morphism. To prove
that the complexes $A$ and $B$ are isomorphic,
we must therefore construct an automorphism $\varphi$
of the object $A=B\in\Mat(\R^2,\emptyset)$ which
satisfies $\varphi\circ d_A=d_B\circ\varphi$.

Let $\mathcal{T}'\subset\mathcal{D}^c\times\R$ be the tangle
represented by $T'\subset\mathcal{D}^c$.
Let $\alpha\subset T'$ the projection
of the arc of $\mathcal{T}'$ connecting $\{a\}\times\{0\}$ to $\{c\}\times\{0\}$,
and let $c_1,\ldots,c_m$ be the sequence of crossings along $\alpha$,
as in Figure~\ref{fig:dotmigration}. As in Subsection~\ref{subs:dotmigration},
we denote $h_1,\ldots,h_m$ the dot migration homotopies
(Definition~\ref{def:dmigration}) and by $X_1,\ldots,X_{m+1}$
the maps $X_k:=X_{p_k}$ (Definition~\ref{def:dmultiplication1}).
For $k=1,\ldots,m$, we define $\varphi_k$ to
be the endomorphism of $A=B\in\Mat(\R^2,\emptyset)$ given by
$$
\varphi_k := 1\otimes 1 + (\dd \delta)\otimes h_k
$$
where $\dd$ is the derivative with respect to the
dot (Definition~\ref{def:dderivative}).
\begin{definition} Let $\varphi$ be the composition
$\varphi:=\varphi_1\circ\ldots\circ \varphi_m\in
\operatorname{End}_{\Mat(\R^2,\emptyset)}(A=B)$.
\end{definition}

Using Lemma~\ref{lem:homotopyproperties0}
and the fact that coefficients are in $\Ftwo$, it is easy to check that
$\varphi_k\circ\varphi_k=1\otimes 1$ and
$\varphi_k\circ\varphi_l=\varphi_l\circ\varphi_k$ for
all $k,l$, and hence also $\varphi\circ\varphi=1\otimes 1$.
In particular, $\varphi$ is invertible.

\begin{remark}
Since every self-crossing of $\alpha$ appears twice
in the list $c_1,\ldots,c_m$, every endomorphism
$\varphi_k$ corresponding to a self-crossing of $\alpha$
appears twice in $\varphi$.
Since $\varphi_k$ squares
to the identity, we can thus ignore all self-crossings
of $\alpha$, and define $\varphi$ as the product over all
$\varphi_k$ for which $c_k$ is not a self-crossing of $\alpha$.
\end{remark}

To see
that $\varphi$ satisfies $\varphi\circ d_A=d_B\circ \varphi$
as desired, we need several technical lemmas.

\begin{lemma}\label{lem:phicommute}
$\varphi$ commutes with $\delta\otimes 1$.
\end{lemma}
\begin{proof}
Corollary~\ref{cor:ddcommute} tells us that
$\dd \delta$ commutes with $\delta$, and this immediately
implies that each $\varphi_k$ (and hence also $\varphi$) 
commutes with $\delta\otimes 1$.
\end{proof}

\begin{lemma}\label{lem:phimigration}
$\varphi_k\circ (1\otimes d)\circ\varphi_k^{-1} =
1\otimes d + (\dd\delta)\otimes (X_k + X_{k+1})$.
\end{lemma}

\begin{proof}
Direct calculation using
$\varphi_k=\varphi_k^{-1}=1\otimes 1 +(\dd d)\otimes h_k$
yields
$$
\varphi_k\circ (1\otimes d)\circ\varphi_k^{-1}
=1\otimes d +  (\dd\delta)\otimes (d\circ h_k + h_k\circ d) +
(\dd\delta)^2\otimes (h_k\circ d \circ h_k)
$$
and now the claim follows from Lemmas~\ref{lem:homotopy} and
\ref{lem:homotopyproperties}.
\end{proof}

\begin{corollary}\label{cor:phimigration}
$\varphi\circ (1\otimes d)\circ\varphi^{-1} =
1\otimes d + (\dd\delta)\otimes (X_a + X_c)$.
\end{corollary}

\begin{proof}
Recall that $X_l=X_{p_l}=x_{p_l}1$ and from this it
easily follows that $X_l\circ h_k=x_{p_l}h_k=h_k\circ X_ö$
for all $k,l$. Thus $\varphi_k$ commutes with $(\dd\delta)\otimes X_l$
for all $k,l$.
Recalling that $\varphi=\varphi^{-1}=\varphi_1\circ\ldots\circ\varphi_m$
and using Lemma~\ref{lem:phimigration} repeatedly, one can now conclude
$$
\varphi\circ (1\otimes d)\circ\varphi^{-1} = d\otimes 1 +
(\dd\delta)\otimes\left[(X_1 + X_2) + (X_2 + X_3) + \ldots + (X_m + X_{m+1})\right]
$$
and the telescope sum in the square brackets
collapses to $X_1+X_{m+1}$ because all
intermediate terms appear twice and hence cancel.
Since $p_1=a$ and $p_{m+1}=c$ (see Subsection~\ref{subs:dotmigration}),
we have $X_1=X_a$ and $X_{m+1}=X_c$, whence
$X_1+X_{m+1}=X_a+X_c$.
\end{proof}

We are now ready to prove Proposition~\ref{prop:main}.

\begin{proof}[Proof of Proposition~\ref{prop:main}]
We have to show that
$\varphi \circ d_A\circ \varphi^{-1} = d_B$.
This is now a direct calculation:
$$
\begin{array}{rcl}
\varphi \circ d_A\circ\varphi^{-1} &\stackrel{(1)}{=}&
\varphi \circ\left(\delta\otimes 1 +
1\otimes d\right)\circ\varphi^{-1}\\
&\stackrel{(2)}{=}&
\delta\otimes 1 + \varphi\circ (1\otimes\delta)\circ\varphi^{-1}\\
&\stackrel{(3)}{=}&
\delta\otimes 1 + 1\otimes d + (\dd \delta)\otimes (X_a + X_c)\\
&\stackrel{(4)}{=}&
\delta\otimes 1 + 1\otimes d +
\left((x_a+x_c)(\dd \delta)\right)\otimes 1\\
&\stackrel{(5)}{=}&
R_{\bullet}(\delta)\otimes 1 + 1\otimes d\\
&\stackrel{(6)}{=}& d_B\,.
\end{array}
$$
Equalities (1) and (6) are the
definitions of $d_A$ and $d_B$, respectively.
Equality (5) is the definition of $R_{\bullet}$.
Equality (2) follows because $\varphi$ commutes with
$\delta\otimes 1$ by Lemma~\ref{lem:phicommute}.
Equality (3) is Corollary~\ref{cor:phimigration}.
To see (4), observe that
$1\otimes X_a=X_a\otimes 1$
because $1\otimes X_a$ and $X_a\otimes 1$
are both obtained from the identity morphism $1\otimes 1$
by inserting a dot into 
the line segment $\{a\}\times [0,1]$
(cf. Definitions~\ref{def:dmultiplication} and
\ref{def:dmultiplication1}).
Therefore:
$$
(\dd \delta)\otimes X_a=
(1\otimes X_a)\circ \left[(\dd \delta)\otimes 1\right]
= (X_a\otimes 1)\circ \left[(\dd \delta)\otimes 1\right]
= (x_a\dd \delta)\otimes 1\,,
$$
and similarly: $(\dd \delta)\otimes X_c=(x_c\dd \delta)\otimes 1$.
\end{proof}


\noindent
{\bf Acknowledgments.} I would like to thank Dror Bar-Natan
for many valuable discussions, in which he
generously shared his ideas on mutation invariance with me.
I would also like to thank Dror for letting
me use a figure from his paper \cite{barnatan-2005-9}
for my Figure~\ref{fig:relations}.
This work was supported by fellowships of the
Swiss National Science Foundation
and of the Fondation Sciences math\'ematiques de Paris.
\bibliography{mutation1}
\end{document}